\documentclass[11pt,a4paper,leqno,twoside]{amsart}%
\usepackage{amsmath,amssymb,amsthm,enumerate,hyperref,color}
\usepackage{amsfonts,amssymb,amsthm,latexsym,graphicx,layout}
\numberwithin{equation}{section}
\newtheorem{theorem}{Theorem}[section]
\newtheorem{definition}[theorem]{Definition}

\newtheorem{proposition}[theorem]{Proposition}
\newtheorem{lemma}[theorem]{Lemma}
\newtheorem{corollary}[theorem]{Corollary}

\newcommand{\rad}{{\text{\upshape rad}}}

\def\a{{\alpha}}

\newcommand{\R}{\mathbb{R}}

\newcommand{\nod}{\mathop{\mathrm{nod}}}

\newif\ifcomment \commentfalse
\def\commentON{\commenttrue}

\long\outer\def\BC#1\EC{\ifcomment \sloppy \par \# \ldots\dotfill
{\em #1} \dotfill \# \par \fi } \commentON

\newcommand{\remove}[1]{}

\def\sideremark#1{\ifvmode\leavevmode\fi\vadjust{\vbox to0pt{\vss
 \hbox to 0pt{\hskip\hsize\hskip1em
 \vbox{\hsize2.1cm\tiny\raggedright\pretolerance10000
  \noindent #1\hfill}\hss}\vbox to15pt{\vfil}\vss}}}%

\definecolor{cadmiumgreen}{rgb}{0.0, 0.42, 0.24}
\definecolor{darkgreen}{rgb}{0.0, 0.5, 0.2}
\definecolor{purple}{rgb}{0.5, 0.0, 0.5}

\title[H\'enon problem]{A note on nonradial nodal solutions to the H\'enon problem in the disc}
\thanks{This work was partially supported by Gruppo Nazionale per l'Analisi Matematica, la Probabilit\`a e le loro Applicazioni (GNAMPA) of the Istituto Nazionale di Alta Matematica (INdAM) and Fabbr\\
}
\author[F.~Gladiali, G.~Stegel]{Francesca Gladiali$^*$, Giovanni Stegel$^*$ }
\address{$*$ Dipartimento di Chimica e Farmacia, Universit\`a di Sassari, via Piandanna 4, 07100 Sassari, Italy\\
\texttt{fgladiali@uniss.it}\\
\texttt{stegel@uniss.it}}

\begin{document}
\maketitle

\begin{abstract}
In this paper we consider some nodal solutions of the H\'enon problem in the unit disc with Dirichlet boundary conditions and we show that they are quasiradial, that is to say  they are nonradial, they have two nodal regions and their nodal line does not touch the boundary of the disc.

\

\end{abstract}


\section{Introduction}\label{sec:intro}
This paper concerns with nodal solutions to the 
 H\'enon problem
\begin{equation}\label{H}
\left\{\begin{array}{ll}
-\Delta u = |x|^{\alpha}|u|^{p-1} u \qquad & \text{ in } B, \\
u= 0 & \text{ on } \partial B,
\end{array} \right.
\end{equation}
where $\a\geq 0$, $p>1$ and 
$B$ stands for the unit ball of the plane. Equation \eqref{H} has important applications in physics. It has been derived in \cite{H} in the study of a cluster of stars with a big collapsed object in the origin and it models also steady-state distributions in some diffusion processes, see \cite{ederson-filomena-2}.\\
One way to obtain solutions that change sign is to minimize the 
the Energy functional
\begin{equation}\label{eq:energy}
\mathcal E_p(u)=\frac 12 \int _B |\nabla u|^2-\frac 1{p+1}\int_B |x|^\a |u|^{p+1}\end{equation} 
constrained to the nodal Nehari manifold
\[\mathcal N_{\nod}:=\{v\in H^1_0(B): \text{ s.t. }v^+, \, v^- \neq 0, \ \mathcal E_p'(v)v^+=0, \ \mathcal E_p'(v)v^-=0\}\]
where $\mathcal E'$ denotes the Fr\'echet derivative of $\mathcal E$ 
and $s^+$ ($s^-$) stands for  the positive (negative) part of $s$. 
The nodal Nehari has been introduced in \cite{CCN}, see also \cite{BW}, to produce the so called least energy nodal solutions. In our setting,
since $H^1_0(B)$ is compactly embedded in $L^{p+1}(B)$ 
for every $p>1$, we can infer that
\[\min_{u\in \mathcal N_{\nod}}\mathcal E_p(u)\]
is attained at a nontrivial function $u_p$, which is a weak, but also classical, solution to \eqref{H}, has two nodal regions, which are the connected components of the set $\{x\in B \ : u(x)\neq 0\}$ and satisfies
\begin{equation}\label{eq:morse-due}
m(u_p)=2.\end{equation}
Here $m(u)$ is the Morse index of  a solution $u$ to \eqref{H}, namely 
the maximal dimension of a subspace $X\subseteq H^1_0(B)$ where the quadratic form 
	\begin{equation}\label{eq:Q-u}
	Q_u(\psi):=\int_B\left(|\nabla \psi|^2-p|x|^\a|u|^{p-1}\psi^2\right)\, dx\end{equation}
	is negative definite.\\
Let us explain how these last properties are obtained in \cite{BW}, since we will need to use them in the sequel. Since $u_p$ is a minimum on $\mathcal N_{\nod}$
\[\langle \mathcal E''_p(u_p)\psi,\psi\rangle=Q_{u_p}(\psi)\geq 0\]
for any $\psi$ on the tangent space to $\mathcal N_{\nod}$ at $u_p$ that we denote by $T_{u_p}$, where $\mathcal E''_p(u)$ is the second Fr\'echet derivative of $\mathcal E_p$ at $u$ and $\langle \ , \ \rangle$ is the pairing. Then the quadratic form $Q_{u_p}$ can be negative definite only on the orthogonal to $T_{u_p}$ and, since $\mathcal N_{\nod}$ has codimension $2$, one gets
\[m(u_p)\leq 2.\]
Moreover, $u_p\notin T_{u_p}$ and $u_p^+$ and $u_p^-$ satisfy $\int_B \left|\nabla u_p^\pm\right|^2 dx=\int _B |x|^\a\left|u_p^\pm\right|^{p+1}dx$, so that we have
\[Q_{u_p}(u_p^\pm)=\int_B \left|\nabla u_p^\pm\right|^2 dx-p\int _B |x|^\a\left|u_p^\pm\right|^{p+1}dx=(1-p)\int_B \left|\nabla u_p^\pm\right|^2 dx<0\]
since $u_p^\pm\neq 0$ and $p>1$, showing
\eqref{eq:morse-due}.\\
 Finally $u_p$ has $2$ nodal regions because 
 \begin{equation}\label{eq:regioni-morse}
 2\leq n(u_p)\leq m(u_p)=2\end{equation}
 if $n(u)$ denotes the number of nodal regions of $u$. The second inequality holds since if $\Omega_p$ is a nodal region of $u_p$ then, letting $z_p=u_p\chi_{\Omega_p}$, where $\chi_\Omega$ is the characteristic function of $\Omega$, then $z_p$ satisfies $\int_B  \left|\nabla z_p\right|^2 dx=\int _B |x|^\a\left|z_p\right|^{p+1}dx$ and as before $Q_{u_p}(z_p)=(1-p) \int_B  \left|\nabla z_p\right|^2 dx<0$. 
 \\
 Moreover, when $\a=0$, letting $\mathcal Z_u:=\{x\in B: u(x)=0\}$ the nodal set of a solution $u$, then
 \[\overline{\mathcal Z_{u_p}}\cap \partial B\neq \emptyset\]
by \cite[Theorem 1.2]{PW} or \cite{AP} and it is reasonable to conjecture that the same holds when $\a>0$.
The same minimization method to produce nodal solutions can be repeated in subspaces of $H^1_0(B)$ which are invariant by the action of some subgroup $\mathcal G$ of the orthogonal group $O(2)$, producing by the principle of symmetric criticality in \cite{Palais} solutions to \eqref{H} invariant by the action of $\mathcal G$.\\
In particular, letting $H^1_{0,\rad}$ the subspace given by radial functions (which are invariant by the action of $O(2)$), and $\mathcal N_{\nod}^{\rad} :=\mathcal N_{\nod}\cap H^1_{0,\rad}$ we can say that 
\[\min_{u\in \mathcal N_{\nod}^{\rad}}\mathcal E_p(u)\]
is attained at a nontrivial function $u_p^\rad$ which is a radial solution to \eqref{H}, has two nodal regions and satisfies
\begin{equation}\label{eq:morse-due-radial}
m_\rad(u_p^\rad)=2\end{equation}
if $m_\rad(u)$ denotes the Morse index in the space $H^1_{0,\rad}$. As in the previous case the last estimate can be deduced from the minimality of $\mathcal E_p(u_p^\rad)$ in $\mathcal N_{\nod}^{\rad}$, while no estimate can be deduced on the total Morse index $m(u_p^\rad)$ by this minimality in $H^1_{0,\rad}$. Moreover obviously $\overline{\mathcal Z_{u_p^\rad}}\cap \partial B=\emptyset.$
\\
Of course, in principle $u_p^\rad$ can coincide with the least energy nodal solution $u_p$ found already, but it is known by 
\cite{AG-19-part2} (Theorem 1.1, formula (1.8) with $N=2$)
that this is not the case because of the estimate
\[m(u_p^\rad)\geq 4+2\left[\frac \a2\right]\]
where $[\cdot]$ is the integer part, 
which contradicts \eqref{eq:morse-due} and shows that $u_p\neq u_p^\rad$ for every $p>1$ and every $\a\geq 0$. A similar estimate 
has been previously deduced, in a different way, in \cite{AP} for autonomous nonlinearities, namely for $\a=0$, see also \cite{DSP} for some similar estimates when $\a\neq 0$.
\\
However, as in the paper \cite{AGN=2}, for every $n\geq 1$ we can consider also the subgroups $\mathcal G_n$ of $O(2)$ generated by any rotation of angle $\frac {2\pi}n$ centered at the origin. Then, letting
\[H^1_{0,n}:=\{v\in H^1_0(B): v(x)=v(g(x)) \text{ for any }x\in B, \text{ for any }g\in \mathcal G_n\}\]
we can repeat the minimization of 
$\mathcal E_p$ on the constraint $\mathcal N_{\nod} ^n:=\mathcal N_{\nod}\cap  H^1_{0,n}$ obtaining that
\[\min_{u\in \mathcal N_{\nod} ^n}\mathcal E_p(u)\]
is attained at a nontrivial function $u_p^n\in H^1_{0,n}$ which solves \eqref{H}, changes sign and satisfies
\begin{equation}\label{eq:morse-due-n}
m_n(u_p^n)=2\end{equation}
if $m_n(u)$ denotes the Morse index in the space $H^1_{0,n}$. As previously observed \eqref{eq:morse-due-n} is due to the fact that $u_p^n$ minimizes $\mathcal E_p$ on $\mathcal N_{\nod} ^n$ and satisfies $\int_B\left|\nabla (u_p^n)^\pm\right|^2=\int_B\left|(u_p^n)^\pm\right|^{p+1}$. We will refer to these functions as nodal $n$-invariant  least energy solutions to \eqref{H}.\\
Anyway it is not clear if by minimizing $\mathcal E_p(u)$ on $\mathcal N_{\nod}^n$ we bring on new solutions, since $u_p^n$ can coincide either with $u_p$ or with $u_p^\rad$ or with $u_p^m$ for $n\neq m$ and, of course, for $n=1$ $u_p^1$ corresponds to $u_p$ since $\mathcal G_1$ is the trivial subgroup and $H^1_{0,1}=H^1_0(B)$. 
But, by \cite{BWW} we know that least energy solutions are foliated Schwarz symmetric, namely axially symmetric with respect to an axis passing through the origin and nonincreasing in the polar angle from this axis. In particular in \cite{PW}, for $p>2$, it is shown that they are strictly decreasing in the polar angle when nonradial.\\
This last result proves then that $u_p^n$ differs from $u_p$ for every $n>1$. 
Moreover, very recently, the question of the extension of the foliated Schwarz symmetry to the case of functions invariant by the action of $\mathcal G_n$ has been raised in \cite{Gladiali-19}. Here, denoting by 
$S _{\frac {2\pi}n}:=\{(x,y)\in B: x>0, y>0, 0<\frac yx< \frac {2\pi}n\}$ and $S_{\frac \pi n}:=\{(x,y)\in B: x>0, y>0, 0<\frac yx< \frac {\pi}n\}$ and $\mathcal B:=\{ (x,y)\in B: x>0, y>0, \frac yx= \frac {\pi}n\}$ the bisector of $S_{\frac {2\pi}n}$, it has been proved that:
\begin{theorem}[\cite{Gladiali-19}]\label{teo1}
Let $u_p^n\in H^1_{0,n}$ be a solution to \eqref{H} with $p\geq 2$ such that 
\[m_n(u_p^n)\leq 2.\]
Then, either $u_p^n$ is radial or $u_p^n$, up to a rotation, is symmetric with respect to $\mathcal B$ in the sector $S _{\frac {2\pi}n}$  and is strictly decreasing in the angular variable in the semi-sector $S_{\frac \pi n}$.
\end{theorem}
\noindent As an application of Theorem \ref{teo1} we get 
that nodal least energy solutions $u_p^n$ either are 
radial or, up to a rotation, are strictly decreasing in the polar angle in  $S_{\frac \pi n}$, showing that $u_p^n\neq u_p^m$ for $n\neq m$ when they are nonradial and $p\geq 2$. 
In \cite{Gladiali-19} and also in \cite{PW} the assumption $p\ge 2$ arises from a convexity request for the nonlinear term and cannot be removed.\\
There remains the possibility that the nodal solutions constructed in the spaces $H^1_{0,n}$ are radial when $n\geq 2$ and we want to know when $u_p^n\neq u_p^\rad$. This is a very difficult problem and it is not possible to give an answer to this question in such a general formulation. Indeed this issues strongly depends on the values of the parameters that describe the problem, namely on $\a$, $n$  and on $p$ so that changing one of this parameter makes the answer change.\\
Nevertheless a positive answer can be given at least for some values of $n$ when the exponent $p$ is large and indeed in  \cite{AGN=2} the following result has been showed:
\begin{theorem}[\cite{AGN=2}]\label{teo2}
	Let $\a\geq 0$ be fixed. There exists an exponent $p^*=p^*(\a)$ such that problem \eqref{H} admits at least $\lceil \frac {2+\a}2 \kappa-1\rceil $ distinct nodal nonradial solutions for every $p >p^*(\a)$.
\end{theorem} 
\noindent Here 
$\lceil t\rceil = \min\{ k \in {\mathbb Z} \, : \, k\ge t\}$ stands for the ceiling function and if $\bar t$ is the unique root of the equation $2\sqrt e\log t+t=0$ then $\kappa=1+\frac {2\sqrt e}{\bar t} \approx 5.1869$. 
When 
$\a=0$ Theorem \ref{teo2} provides $\lceil \kappa-1\rceil=5 $ nodal nonradial solutions, namely $u_p^1=u_p$, $u_p^2$, $u_p^3$, $u_p^4$ and $u_p^5$ and gives back a previous multiplicity result in \cite{GI} obtained considering similar, but slightly different spaces $H^1_{0,n}$. In any case Theorem \ref{teo1} implies that the solutions in \cite{GI} for $\a=0$ coincide with the ones of Theorem \ref{teo2}. \\
Then, by Theorem \ref{teo2} we have $\lceil \frac {2+\a}2 \kappa-1\rceil $ different nodal nonradial solutions to \eqref{H} when $p$ is large enough that are given by  the nodal least energy $n$-invariant solutions $u_p^1, u_p^2,\dots, u_p^{\lceil \frac {2+\a}2 \kappa-1\rceil }$.\\
Starting from these solutions we want here to study the properties of the nodal sets of $u_p^n$.
In this symmetric setting, indeed, the inequality \eqref{eq:regioni-morse} which relates the number of nodal regions of $u_p^n$ to its Morse index is no longer that clear. Due to the rotations invariance it is enough to consider any function $u\in H^1_{0,n}$ in a sector $S$ of angle $\frac {2\pi}n$.
Denoting by $\widetilde n(u_p^n)$ then 
the number of the nodal regions of $u_p^n$ in $S$ it can be easily derived that 
\[2\leq \widetilde n(u_p^n)\]
since $u_p^n$ changes sign. But the other inequality $\widetilde n(u_p^n)\leq m^n(u_p^n)=2$ does not hold any more since the situation depicted in Fig \ref{fig1} is also possible. This is why we want to investigate here these questions:\\

\begin{figure}[t!]
\includegraphics[scale=2]{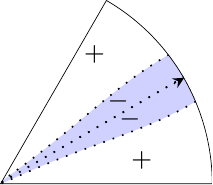}
\caption{The sector $S_{\frac{2\pi}n}$ and one possible nodal configuration}
\label{fig1}
\end{figure}

\noindent 
{\em How many nodal domains do the nodal least energy solutions $u_p^n$ have when they are nonradial?
Does the closure of their nodal set touch the boundary of $B$? Is the nodal set a regular curve?
Which are the possible shapes of their nodal regions?}

\

\noindent We will see that the answer strongly depends on the degree of symmetry of the solution, namely on the values of $n$. To explain the different possibilities that can arise we say that a nodal region is {\em $n$-invariant} if it is invariant by the action of $\mathcal G_n$. Of course a nodal region $n$-invariant is not contained in any sector of angle $\frac {2\pi}n$.
As a consequence of the strict angular monotonicity  in Theorem \ref{teo1} we will see that only 
the following possibilities hold for a least energy nodal solution $u_p^n$ when it is nonradial:\\
{\em case 1)} \  $u_p^n$ admits $2n$ nodal regions in $B$. In this case there exists a connected component of $\mathcal Z_{u_p^n}$ that contains the origin and whose closure intersects $\partial B$, see Fig.1. \\
{\em case 2)} \  $u_p^n$ admits $n+1$ nodal regions in $B$. In this case there exists a sector of angle $\frac{2\pi}n$ that contains a nodal region of $u_p^n$ while the other nodal region is connected and $n$-invariant, see Fig.2.\\
{\em case 3)} \   $u_p^n$ admits $2$ nodal regions in $B$ which are connected, $n$-invariant and the closure of  the nodal set of $u_p^n$ does not touch the boundary of $B$, see Fig.3. 

\begin{figure}[t!]
\includegraphics[scale=2]{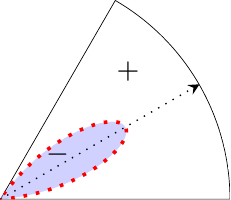}
\quad \quad \quad \quad \quad
\includegraphics[scale=2]{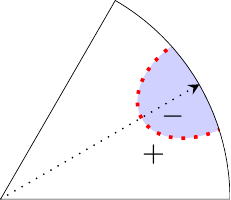}
\caption{Two possible nodal configurations when case 2) occurs}
\label{fig2}
\end{figure}

\noindent In principle some other configurations are possible but they are ruled out by the symmetry and the monotonicity of $u_p^n$ given by Theorem \ref{teo1} and the Morse index estimate in \eqref{eq:morse-due-n}. We can then say that a lower Morse index implies a smaller complexity in the geometry of the nodal configuration of the solutions and this is also true in symmetric spaces. We can then introduce the following definition: 
\begin{definition}
We say that a solution $u$ is quasiradial if it is nonradial, it has only two nodal regions and the closure of its nodal set does not touch the boundary.
\end{definition}
\noindent Of course $u_p^n$ is quasiradial only when case $3)$ happens.

\begin{figure}[t!]
\includegraphics[scale=2]{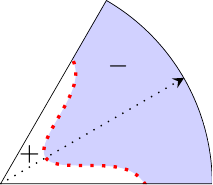}
\caption{A possible nodal configuration when case 3) occurs}
\label{fig3}
\end{figure}

\

In this paper we try to understand the possibile shapes of the nodal zones when $p$ is large and in particular we can prove the following:
\begin{theorem}\label{teo3}
Let $p>p^*$ and let $u_p^n$ be a nodal least energy $n$-invariant nonradial solution as in Theorem \ref{teo2}, for $n=1, \dots, \lceil \frac {2+\a}2 \kappa-1\rceil$. Then, only the possibilities of case $1)$, $2)$ and $3)$ can hold. 
Moreover $u_p^n$ can be of type $1)$ if and only if $n\leq \left[\frac {2+\a} 4 \gamma\right]$ where $\gamma \approx 4.859$ and $\left[\cdot \right]$ is the integer part. It can be of type $2)$ if  and only if $n\leq  \left[\frac {2+\a} 2 \gamma-1\right] $. Finally for $n>  \left[\frac {2+\a} 2 \gamma-1\right] $  $u_p^n$ is of type $3)$, $\mathcal Z_{u_p^n}$ does not intersect $\partial B$ and $u_p^n$ is quasiradial. 
\end{theorem} 
\noindent See Proposition \ref{prop-6} for the characterization of the constant $\gamma$. 
As simple corollaries we obtain:
\begin{corollary}\label{cor-fin}
For every $\a\geq 0$ problem \eqref{H} admits at least $2$ quasiradial solutions when $p$ is large enough whose nodal set is a smooth curve.
\end{corollary}

\begin{corollary}\label{cor-fin2}
The number of quasiradial solutions of problem \eqref{H} in Corollary \ref{cor-fin} increases in $\a$ and goes to infinity as $\a\to \infty$.
\end{corollary}

The results are obtained by comparing the energy of the solutions $u_p^n$ with the energy of the radial solution $u_p^\rad$ and this approach has been previously used in \cite{DeIP13} in a different setting. In particular we prove that each nodal region carries a minimum amount of energy which adds up to the energy of the other regions. Nevertheless, since, by construction, $u_p^n$ is a minimum on the Nehari in $H^1_{0,n}$, this sum should be smaller than the energy of the radial solution. Finally the monotonicity in Theorem \ref{teo1}, together with the Morse index estimate \eqref{eq:morse-due-n} bound the number of the nodal regions and of the nodal shape so that only the possibilities 1), 2) and 3) can occur.

\section{Proof of the results}\label{se:2}
\noindent First we recall from \cite{RW} a useful estimate for functions in $H^1_0(\Omega)$. Here by $\|\cdot \|_q$ we mean the norm in $L^q$. 
\begin{lemma}\label{lem-RW}
For every $t\geq 2$ there is $D_t$ such that
\[\|v\|_{t}\leq D_t \sqrt t \|\nabla v\|_2\]
for all $v\in H^1_0(\Omega)$ where $\Omega$ is a bounded domain of $\R^2$. Furthermore
\begin{equation}\label{eq:limit-D}
\lim_{t\to \infty}D_t=(8\pi e)^{-\frac 12}.
\end{equation}
\end{lemma}
\noindent Next we apply the previous lemma to functions that belong to the Nehari manifold and we get the following estimate:
\begin{lemma}\label{lem-2}
Let $\Omega\subseteq B$ and let 
$w_p\in H^1_0(\Omega)$ be such that 
\begin{equation}\label{eq:nehari}
\int_{\Omega}|\nabla w_p|^2=\int_{\Omega}|x|^\a|w_p|^{p+1}
\end{equation}
for every $p$.
Then 
\begin{equation}\label{eq:stima-asintotica}
\liminf_{p\to \infty}p\int_{\Omega}|\nabla w_p|^2\geq 8\pi e.
\end{equation}
\end{lemma}
\begin{proof}
By Lemma \ref{lem-RW} with $t=p+1$
we have
\[\int_{\Omega}|\nabla w_p|^2\geq\frac{\left(\int_{\Omega}| w_p|^{p+1}\right)^{\frac 2{p+1}}}{(p+1)D_{p+1}^2}\geq \frac{\left(\int_{\Omega}|x|^\a| w_p|^{p+1}\right)^{\frac 2{p+1}}}{(p+1)D_{p+1}^2}
\]
and, using \eqref{eq:nehari}
\[\left(\int_{\Omega}|\nabla w_p|^2\right)^{\frac {p-1}{p+1}}\geq \left((p+1)D_{p+1}^2\right)^{-1}\]
which gives
\[
p\int_{\Omega}|\nabla w_p|^2\geq p\left((p+1)D_{p+1}^2\right)^{-\frac{p+1}{p-1}}\]
so that \eqref{eq:stima-asintotica} follows using \eqref{eq:limit-D} and passing to the $\liminf$.
\end{proof}
\noindent We can now obtain an estimate of the $L^2$ norm of the gradient of a solution $u_p$ corresponding to each nodal zone, namely:
\begin{lemma}\label{lem-3}
Let $u_p$ be a solution to \eqref{H} and let $\Omega_p\subset B$ be a nodal region of $u_p$. Then
\begin{equation}\label{eq:stima-energy}
\liminf_{p\to \infty}p\int_{\Omega_p}|\nabla u_p|^2\geq 8\pi e.
\end{equation}
\end{lemma}
\begin{proof}
Let $z_p:=u_p\chi_{\Omega_p}$ where $\chi_\Omega$ denotes the characteristic function of $\Omega$. Then $z_p\in H^1_0(B)$ and $z_p\equiv 0$ in $B\setminus\Omega_p$ so that
\[\int_B|\nabla z_p|^2=\int_{\Omega_p}|\nabla u_p|^2 \  \text{ and } \ \int_{B}|x|^\a|z_p|^{p+1}=\int_{\Omega_p}|x|^\a|u_p|^{p+1}\]
Multiplying \eqref{H} by $z_p$ and integrating in $B$ we get
\[\int_{\Omega_p}|\nabla u_p|^2=\int_B \nabla u_p\nabla z_p=\int_B|x|^\a |u_p|^{p-1}u_pz_p=\int_{\Omega_p}|x|^\a |u_p|^{p+1}
\]
from which it follows that $z_p$ satisfies \eqref{eq:nehari} in $B$. By the previous lemma then 
\[\liminf _{p\to \infty}p\int_B|\nabla z_p|^2\geq 8\pi e.\]
\end{proof}
\noindent It is easy now to obtain an estimate of the energy $\mathcal E_p$ in every nodal region. Indeed
\begin{lemma}\label{lem-4}
Let $u_p$ be a solution to \eqref{H} and let $\Omega_p$ be a nodal region of $u_p$. Then
\begin{equation}\label{eq:stima-energia-basso}
\liminf_{p\to \infty}p\left(\frac 12 \int_{\Omega_p}|\nabla u_p|^2-\frac 1{p+1}\int_{\Omega_p}|x|^\a | u_p|^{p+1}\right)\geq 4\pi e.
\end{equation}
\end{lemma}
\begin{proof}
We already know from the previous lemma that 
\[\int_{\Omega_p}|\nabla u_p|^2=\int_{\Omega_p} |x|^\a | u_p|^{p+1}\]
so that the energy in \eqref{eq:stima-energia-basso} becomes
\[\begin{split}
&\liminf_{p\to \infty}p\left(\frac 12 \int_{\Omega_p}|\nabla u_p|^2-\frac 1{p+1}\int_{\Omega_p}|x|^\a | u_p|^{p+1}\right) =\liminf_{p\to \infty} \frac {p-1}{2(p+1)}p\int_{\Omega_p}|\nabla u_p|^2\\
&=\frac 12 \liminf_{p\to \infty} p\int_{\Omega_p}|\nabla u_p|^2\geq 4\pi e.
\end{split}\]
\end{proof}
\noindent We can use \eqref{eq:stima-energia-basso} to obtain an estimate from below of the energy of a solution $u_p$ given the number of its nodal regions. 
\begin{corollary}\label{cor-5}
Let $u_p$ be a solution to \eqref{H} that has at least $N$ nodal regions for $p$ large. Then
\begin{equation}\label{eq:stima-energia-basso-totale}
\liminf_{p\to \infty}p\mathcal E_p(u_p)
\geq 4\pi eN.
\end{equation}
\end{corollary}
\begin{proof}
By assumption $u_p$ has at least $N$ nodal regions that we denote by $\Omega_{1,p},\dots, \Omega_{N,p}$ so that
\[ p\mathcal E_p(u_p)\geq \sum_{j=1}^Np\left(\frac 12 \int_{\Omega_{j,p}}|\nabla u_p|^2-\frac 1{p+1}\int_{\Omega_{j,p}}|x|^\a | u_p|^{p+1}\right)\]
Then using the  $\liminf$ properties and estimate \eqref{eq:stima-energia-basso} in every nodal zone $\Omega_{j,p}$ we get
\[\liminf_{p\to \infty} p\mathcal E_p(u_p)\geq \sum_{j=1}^N \liminf_{p\to \infty}p\left(\frac 12 \int_{\Omega_{j,p}}|\nabla u_p|^2-\frac 1{p+1}\int_{\Omega_{j,p}}|x|^\a | u_p|^{p+1}\right)
\geq \sum_{j=1}^N 4\pi e\]
concluding the proof.
\end{proof}
\noindent We conclude this first part with an estimate of the energy of the radial solution $u_p^\rad$. We deduce it from the estimate on the energy of radial nodal least energy solutions in \cite{GGP14}, using a transformation that relates radial solutions of the two problems introduced in \cite{GGN} and \cite{GGN2}. 
\begin{proposition}\label{prop-6}
Let $u_p^\rad$ be a radial solution to \eqref{H} with two nodal zones. Then
\begin{equation}\label{eq:energia-radiale}
\lim_{p\to \infty}p\mathcal E_p(u_p^\rad)=2(2+\a)\gamma \pi e
\end{equation}
where $\gamma$  is approximately equals to $4.859$ and is given by $\gamma=e^{-\frac{\sqrt e}{\bar t+\sqrt e}}\left(\frac e{\bar t^2}+1+\frac{2\sqrt e}{\bar t}\right)$ where $\bar t$ is the unique root of the equation $2\sqrt e \log t+t=0$.
\end{proposition}
\begin{proof}
Letting $v_p(t)=\left(\frac 2{2+\a}\right)^{\frac 2{p-1}}u_p^\rad (r)$, for $t=r^\frac{2+\a}2$ and $r=|x|$, as in \cite[Sez. 2]{AGN=2} it is easily seen that $v_p(t)$ is a radial nodal solution to 
\begin{equation}\label{LE}\left\{\begin{array}{ll}
-\Delta v _p= |v_p|^{p-1} v_p \qquad & \text{ in } B, \\
v_p= 0 & \text{ on } \partial B,
\end{array} \right.
\end{equation}
with two nodal zones and 
\[\int_B |\nabla u_p^\rad |^2=\left(\frac {2+\a}2\right)^{\frac {p+3}{p-1} } \int_B |\nabla v_p|^2 dx
\]
The limit of the energy associated with $v_p$ has been studied in \cite{GGP14} where it is proved that
\[\lim_{p\to \infty} p\int_B |\nabla v_p|^2 dx= 8\pi \gamma e\]
Then \eqref{eq:energia-radiale} follows recalling that 
\[\lim_{p\to \infty} p\mathcal E_p(u_p^\rad)=\frac 12 \lim_{p\to \infty} \left(\frac {2+\a}2\right)^{\frac {p+3}{p-1} }
p\int_B |\nabla v_p|^2 dx.\]
\end{proof}
\noindent The constant $\gamma$ has been characterized in \cite[Theorem 2]{GGP14} and we 
refer the reader to that paper in order to better understand where it comes from.

\

\noindent Now we turn to the least energy $n$-invariant solutions $u_p^n$ and we prove an energy estimate.
\begin{lemma}\label{lem-7}
Let $u_p^n$ be a least energy $n$-invariant solution to \eqref{H}. Then, for every $n\geq 1$
\begin{equation}\label{eq:sima-energy-n}
\limsup_{n\to \infty}p\mathcal E_p(u_p^n)\leq 2(2+\a)\gamma \pi e.
\end{equation}
\end{lemma}
\begin{proof}
It easily follows since, by construction, 
\[p\mathcal E_p(u_p^n)\leq p\mathcal E_p(u_p^\rad).\]
\end{proof}
\noindent We are now able to prove an estimate on the number of nodal regions that a least energy $n$-invariant solution can have for any value of $n$, namely:
\begin{proposition}\label{prop-9}
Let $u_p^n$ be a least energy $n$-invariant nodal solution to \eqref{H}. Then $u_p^n$  has at most $N_\a:=\left[\frac {2+\a}2\gamma\right]$ nodal regions for $p$ large, where $\left[ \cdot \right]$ stands for the integer part and $\gamma$ is as in Proposition \ref{prop-6}
\end{proposition}
\begin{proof}
When $u_p^n$ is radial we are done since it has $2<N_\a$ nodal regions. When $u_p^n$ is nonradial let $N$ be the number of its nodal regions as $p\to \infty$. Equation \eqref{eq:stima-energia-basso-totale} implies
\[\liminf_{p\to \infty}p\mathcal E_p(u_p^n)\geq 4\pi e N\]
which together with \eqref{eq:sima-energy-n} implies that
\[4\pi e N\leq \liminf_{p\to \infty}p\mathcal E_p(u_p^n)\leq \limsup_{p\to \infty}p\mathcal E_p(u_p^n)\leq 2(2+\a)\gamma \pi e
\]
showing that 
\[N\leq \frac {2+\a}2\gamma.\]
\end{proof}
\noindent Thanks to Proposition \ref{prop-9} we are in position to obtain some properties of the nodal configurations of $u_p^n$ in order to get Theorem \ref{teo3}.\\ 
We start with a bound on the possible number of nodal regions inside a sector $S$ of angle $\frac {2\pi}n$. 
Clearly there exists an angle $0<\varphi<2\pi$ such that $S=R_\varphi (S_{\frac {2\pi}n})$, where $R_\varphi$ denotes a counterclockwise rotation of angle $\varphi$ centered at the origin.

\begin{lemma}\label{lem-9}
Let $u_p^n$ be a least energy $n$-invariant nodal solution to \eqref{H}. 
There can be at most two nodal regions of $u_p^n$ strictly contained in a sector $S=R_\varphi (S_{\frac {2\pi}n})$.
\end{lemma}
\begin{proof}
Assume by contradiction that there exists a sector $S$ of angle $\frac {2\pi}n$ that contains $\ell>2$ distinct nodal regions of $u_p^n$, that we denote by $D_1, D_2, \dots, D_\ell$. We denote then by $D_i^n$ the subset of $B$ obtained by $D_i$ through subsequent rotations of angle $\frac {2\pi}n$, namely $D_i^n:=D_i\cup R_{\frac {2\pi}n}\left(D_i\right)\cup R_{2\frac {2\pi}n}\left(D_i\right)\cup \dots\cup R_{(n-1)\frac {2\pi}n}\left(D_i\right)$. Of course $D_i^n$ is invariant by the action of $\mathcal G_n$ and it has $n$ connected components since we are assuming $D_i\subset S$. 
Next, we let $z_i:=u_p^n\chi_{D_i^n}$ for $i=1,\dots,\ell$ and we observe that $z_i\in H^1_0(B)$, $z_i$ is $n$-invariant, since $D_i^n$ and $u_p^n$ are $n$-invariant, $z_i\neq 0$, $z_i\equiv 0$ in $B\setminus\{D_i^n\}$ and, as in the proof of Lemma \ref{lem-3},  $z_i$ satisfies \eqref{eq:nehari} in $B$. 
We can then infer that
\[Q_{u_p^n}(z_i)=\int_B|\nabla z_i|^2-p\int_B|x|^\a|z_i|^{p+1}=(1-p)\int_B|\nabla z_i|^2<0\]
for $i=1,\dots,\ell$. We have constructed so far $\ell$ functions, $z_1, \dots,z_\ell\in H^1_{0,n}$ which are orthogonal in $L^2(B)$, since they have disjoint supports, and that make negative the quadratic form $Q_{u_p^n}$. This  contradicts \eqref{eq:morse-due-n} and shows that $\ell\leq 2$.
\end{proof}

\noindent As a consequence of the previous proof
we immediately have:
\begin{corollary}\label{cor1}
Let $u_p^n$ be a least energy $n$-invariant nodal solution to \eqref{H}. Then $u_p^n$ can have only two $n$-invariant nodal components that can be connected or not.
\end{corollary}

\noindent We can also prove the following result:

\begin{lemma}\label{lem-10}
Let $u_p^n$ be a least energy $n$-invariant nodal solution to \eqref{H}. Suppose there 
exists a sector $S$ of angle $\frac {2\pi}n$ that contains two nodal regions of $u_p^n$, $D_1$, $D_2$. 
Then $S\setminus \mathcal Z_{u_p^n}=D_1\cup D_2$ and $u_p^n$ admits $2n$ nodal components in $B$.
\end{lemma}
\begin{proof}
Assume,  by contradiction, that $\left(B\setminus \mathcal Z_{u_p^n}\right) \cap \left(S \setminus\{D_1\cup D_2\}\right)=\widetilde D\neq \emptyset$. 
The previous lemma implies that any sector $S$ can contain at most two nodal regions of $u_p^n$, meaning that $\widetilde D$ is not a connected component of $B\setminus \mathcal Z_{u_p^n}$ but it is contained in a connected component $D$ of $B\setminus \mathcal Z_{u_p^n}$. If $D$ is $n$-invariant we are done, else we let, as in the previous lemma $D^n$ be the subset of  $B\setminus \mathcal Z_{u_p^n}$ which contains $D$ and is $n$-invariant. We also denote by $D_i^n$ for $i=1,2$, the subsets of $B$ obtained by $D_i$ through subsequent rotations of angle $\frac {2\pi}n$, so that they are $n$-invariant. Obviously $D^n\cap D_i^n=\emptyset$ for $i=1,2$ and this contradicts Corollary \ref{cor1}. \\
Finally, since $\left(B\setminus \mathcal Z_{u_p^n} \right)\cap S =D_1\cup D_2$ then by the rotation invariance of $u_p^n$ it easily follows that $B\setminus  \mathcal Z_{u_p^n}$ admits $2n$ components. 
\end{proof}

\noindent Now we use the monotonicity in Theorem \ref{teo1} to obtain some useful properties of $u_p^n$. We will assume tacitly hereafter that $u_p^n$ is strictly decreasing in the angular variable in the sector $S_{\frac {\pi}n}$.  First we show that:

 \begin{lemma}\label{lem-11}
 Let $u_p^n$ be a least energy $n$-invariant nonradial nodal solution to \eqref{H}. Assume $u_p^n(\bar x)\leq 0$ for $\bar x\in \partial S_{\frac{2\pi}n}\cap \{(x,y): 0<x<1, y=0\}$. Then  $u_p^n(x)<0$ in $\{x\in S_{\frac{2\pi}n}: |x|=|\bar x|\}$.
 \end{lemma}
 \begin{proof}
It follows by the strict angular monotonicity of $u_p^n$ in $S_{\frac {\pi}n}$ and in $S_{\frac{2\pi}n}\setminus \overline{S_{\frac {\pi}n}}$ given by Theorem \ref{teo1}.
 \end{proof}

\begin{lemma}\label{lem11-bis}
Let $u_p^n$ be a least energy $n$-invariant nodal solution to \eqref{H}. Then it cannot have critical points on $\mathcal Z_{u_p^n}\cap S_{\frac {2\pi}n}$.
\end{lemma}
\begin{proof}
If $u_p^n$ is radial then $\mathcal Z_{u_p^n}$ is a circle of radius $r<1$ that separates the two nodal regions of $u_p^n$ and there cannot be critical points on it, due to the Hopf Lemma. When $u_p^n$ is nonradial instead, by Theorem \ref{teo1} we know that, up to a rotation, its critical points in $S_{\frac {2\pi}n}$ lie on the bisector $\mathcal B$. We can then assume, by contradiction, that there exists one critical point $\bar x$ on $\mathcal B\cap \mathcal Z_{u_p^n}$. We recall first that if a point $\bar x$ belongs to the nodal set, then there exists a positive radius $r$ such that $\{(u_p^n)^{-1}(0)\}\cap  B(\bar x , r)$ is made of $2k$ $C^1$-simple arcs, for some
integer $k\geq 1$, which all end in $\bar x$ and whose tangent lines at $\bar x$ divide the disc into $2k$ angles of equal amplitude, see \cite{HW} or \cite[Theorem 2.1]{HHT}. In particular if $\nabla u_p^n(\bar x)=0$,  then $k\geq 2$ and $\{(u_p^n)^{-1}(0)\}\cap  B(\bar x , r)$ is made of at least four $C^1$-simple arcs ending in $\bar x$. By the symmetry and the strict angular monotonicity of $u_p^n$ in $ S_{\frac {2\pi}n}$ the unique possible configuration in this case is the one in Figure 4.

\begin{figure}[t!]
\includegraphics[scale=2]{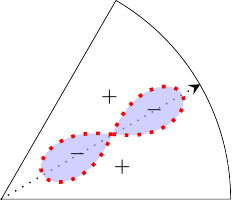}
\quad \quad \quad \quad \quad
\includegraphics[scale=2]{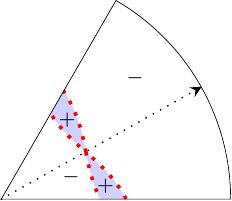}
\caption{Two possible nodal configurations when $u_p^n$ has a critical point on the bisector $\mathcal B$}
\label{fig4}
\end{figure}

\noindent Indeed if $u_p^n(\bar x)=0$ then $u_p^n (x)>0$ for every $x\in \bar S_{\frac {2\pi}n}$ such that $|x|=|\bar x|$ and $x\neq \bar x$, see Lemma \ref{lem-11}. In $\bar x$ the set $\mathcal Z_{u_p^n}$ is made of at least $4$ simple arcs ending in $\bar x$ which are the boundaries or part of the boundaries of the nodal regions of $u_p^n$, that can neighbour each other only if $u_p^n$ has different sign within them. Then $u_p^n$ admits a nodal region contained in $B(0,|\bar x|)$ in which $u_p^n<0$ that we call $D_1$ and another one in which $u_p^n$
 is negative, 
contained instead in $B\setminus \bar B(0,|\bar x|)$, that we denote by $D_2$. These regions are separated by the nodal region in which $u_p^n$ is positive that we call $D_3$. Calling then $D_1^n$, $D_2^n$ and $D_3^n$ the $n$-invariant extensions of $D_1$, $D_2$ and $D_3$ we obtain a contradiction with Corollary \ref{cor1}. So $u_p^n$ does not admit critical points on $\mathcal Z_{u_p^n}\cap S_{\frac {2\pi}n}$. 
\end{proof}
\noindent As corollaries of the previous lemma we have:
\begin{corollary}\label{cor3}
Let $u_p^n$ be a least energy $n$-invariant nonradial nodal solution to \eqref{H}.
The set $\mathcal Z_{u_p^n}$ can contain at most one critical point of $u_p^n$ which is the origin.
\end{corollary}
\begin{proof} By Lemma \ref{lem11-bis} $u_p^n$ does not have critical points on $\mathcal Z_{u_p^n}\cap S_{\frac {2\pi}n}$. It is possible however that it has critical points on $\mathcal Z_{u_p^n}\cap \partial S_{\frac {2\pi}n}\setminus\{O\}$. But then the function $-u_p^n$ is still a least energy $n$-invariant nonradial nodal solution to \eqref{H} such that its rotation of angle $\frac \pi n$ is decreasing in $S_{\frac \pi n}$ and has a critical point on the bisector of $S_{\frac {2\pi}n}$, which is impossible by the previous lemma. Then the only critical point of $u_p^n$ on $\mathcal Z_{u_p^n}$ can be the origin.
\end{proof}

\begin{corollary}\label{cor2}
The set $\mathcal Z_{u_p^n}\setminus\{O\}$ is locally a smooth curve.
\end{corollary}
\noindent It follows by the Implicit Function Theorem.
\begin{corollary}\label{cor3}
If $O\in \mathcal Z_{u_p^n} $ then $O$ is a critical point of order $n$.
\end{corollary}
\begin{proof}
By symmetry $O$ is a critical point for $u_p^n$ and when $O\in \mathcal Z_{u_p^n}$ there exists a positive radius $r$ such that $\{(u_p^n)^{-1}(0)\}\cap  B(O , r)$ is made of $2k$ $C^1$-simple arcs, for some
integer $k\geq 2$. By Theorem \ref{teo1}, we can assume, up to a rotation, that $u_p^n$ is strictly decreasing in the angular variable in $S_{\frac \pi n}$. This implies that in the sector $S_{\frac \pi n}$ there should be one of the $2k$ $C^1$-simple arcs of $\{(u_p^n)^{-1}(0)\}\cap  B(O , r)$. If this is not the case then this arc coincide with the bisector $\mathcal B$, but since the nodal set is the boundary of two consecutive nodal regions in which the sign of $u_p^n$ is opposite, then this case is not possible due to the symmetry of $u_p^n$ with respect to $\mathcal B$, (see Theorem \ref{teo1}). So one arc is contained in $S_{\frac \pi n}$ and, by the symmetry of $u_p^n$, another arc is contained in $S_{\frac{2\pi}n}\setminus \overline{ S_{\frac \pi n}}$. The invariance of $u_p^n$ with respect to $\mathcal G_n$ then implies that $k\geq n$. On the other hand $k>n$ cannot hold, otherwise in the sector $\overline{ S_{\frac \pi n}}$  there should be at least $2$ simple arcs which are the boundaries of nodal regions of $u_p^n$ meaning that in $ \overline{ S_{\frac \pi n}}\cap B(O , r)$ $u_p^n$ changes sign at least $3$ times 
and this is not possible since it is strictly decreasing in the angular variable.
\end{proof}

\begin{lemma}\label{lem-12}
Let $u_p^n$ be a least energy $n$-invariant nodal solution to \eqref{H}. Assume $\mathcal Z_{u_p^n}$ admits a connected component that contains the origin and whose closure intersects $\partial B$. Then $u_p^n$ has $2n$ nodal regions in $B$.
\end{lemma}
\begin{proof}
Of course, under the assumptions $u_p^n$ cannot be radial. Moreover due to Lemma \ref{lem-11} 
$u_p^n>0$ on $\partial S_{\frac{2\pi}n}\cap B \setminus\{0\}$, otherwise $\mathcal Z_{u_p^n}$ would not possess a connected component that contains the origin and whose closure intersects $\partial B$. Then 
there exists at least one component of $S_{\frac{2\pi}n}
\setminus \mathcal Z_{u_p^n}$ 
in which $u_p^n<0$. 
We claim that the set $D_1:=\{x\in S_{\frac{2\pi}n}:u_p^n(x)<0\}$ is made of one connected component. Indeed, if it contained at least $2$ connected components this would contradict Lemma \ref{lem-10}, since there exists at least another connected component of $B\setminus \mathcal Z_{u_p^n}$
that contains the $x$-axis, in which $u_p^n$ is positive. \\
By assumptions and by the symmetry of $u_p^n$ there should be a connected component $\mathcal Z$ of  $\mathcal Z_{u_p^n}$ that contains the origin and whose closure intersects $\partial B$ and, since $u_p^n>0$ on $\partial S_{\frac{2\pi}n}\cap B \setminus\{0\}$, it should be contained in $S_{\frac {2\pi}n}$. Moreover $\mathcal Z\cap \mathcal B=O$ where $\mathcal B$ is, as before, the bisector of $S_{\frac {2\pi}n}$. Indeed, by contradiction, if $\mathcal Z\cap \mathcal B=\bar x\neq O$ then,
 by the symmetry of $u_p^n$, $\bar x$ should be a critical point for $u_p^n$ and this is not possible by Lemma \ref{lem11-bis}. So we can assume $\mathcal Z\subset S_{\frac \pi n}$ and by the symmetry of $u_p^n$ the symmetric of $\mathcal Z$ with respect $\mathcal B$, called $\mathcal Z'$, belongs to $\mathcal Z_{u_p^n}\cap S_{\frac {2\pi}n}\setminus \overline{ S_{\frac \pi n}}$.
 
Then $\bar {\mathcal Z}\cup \mathcal Z'$ shapes the boundary of $D_1$ inside $S_{\frac {2\pi}n}$ which is the unique nodal region of $u_p^n$ contained in $S_{\frac{2\pi}n}$ and $D_1$ separates the regions of $S_{\frac{2\pi}n}$ in which $u_p^n$ is positive.

We call  $D_1^n$ the subset of $B\setminus \mathcal Z_{u_p^n}$ which contains $D_1$ and is $n$-invariant. Since $D_1\subset S_{\frac{2\pi}n}$, then $D_1^n$ is made of $n$ connected components.\\
Next we let $D_2$ be the connected component of $B\setminus \mathcal Z_{u_p^n}$ that contains the set $\{(x,0): 0<x<1\}$ and we denote $D_2^n$ the the subset of $B\setminus \mathcal Z_{u_p^n}$ which contains $D_2$ and is $n$-invariant. 
Of course $\{(x,y)\in B: x>0,y>0, \frac yx=\frac {2\pi}n\}\subset R_{\frac{2\pi}n}(D_2)\subset D_2^n$. If this is not true we can find a contradiction with Lemma \ref{cor1}. Moreover, since $D_2^n=D_2\cup R_{\frac{2\pi}n}(D_2)\cup R_{\frac{4\pi}n}(D_2)\cup\dots\cup R_{\frac{2(n-1)\pi}n}(D_2)$ then $D_2^n$ is made of $n$ connected components. Furthermore, there cannot be other nodal regions since this would contradict Lemma \ref{lem-10}. This shows that the nodal regions are $2n$.
\end{proof}

\begin{corollary}
Let $u_p^n$ be a least energy $n$-invariant nodal solution to \eqref{H}. Assume $\overline{\mathcal Z_{u_p^n}}$ admits in $\overline{S_{\frac{2\pi}n}}$ a connected component $\mathcal C_p$ that contains the origin and intersects $\partial B$. Then $\overline {\mathcal C_p}$ intersect the bisector $\mathcal B$ only in the origin and eventually on $\partial B$.
\end{corollary}

\noindent Now we consider the case in which there exists a sector $S$ of angle $\frac {2\pi}n$ that contains one component of $u_p^n$ but $\mathcal Z_{u_p^n}$ does not admit a connected component that contains the origin and whose closure intersects $\partial B$. The set $\mathcal Z_{u_p^n}$ can contain either the origin or intersect the boundary $\partial B$ or neither. 
That makes no difference in this context and we can prove:
\begin{lemma}\label{lem-14}
Let $u_p^n$ be a least energy $n$-invariant nodal solution to \eqref{H}. Assume that there exists a sector $S$ of angle $\frac {2\pi}n$ that contains one component of $B\setminus \mathcal Z_{u_p^n}$, but $\mathcal Z_{u_p^n}$ does not admit a connected component that contains the origin and whose closure intersects $\partial B$. Then $u_p^n$ has $n+1$ nodal regions in $B$.
\end{lemma}
\begin{proof}
We let $D_1$ be the connected component of 
$B\setminus \mathcal Z_{u_p^n}$ contained in $S$. We can always assume that $u_p^n<0$ in $D_1$ (otherwise we can consider $-u_p^n$ instead of $u_p^n$).
We call  $D_1^n$ the subset of $B\setminus \mathcal Z_{u_p^n}$ which contains $D_1$ and is $n$-invariant. Since $D_1$ is contained in a sector of amplitude $\frac{2\pi}n$, then $D_1^n$ is made of $n$ connected components $D_1\cup R_{\frac{2\pi}n}(D_1)\cup\dots\cup R_{(n-1)\frac{2\pi}n}(D_1)$.\\
We claim that 
$u_p^n>0$ in $D_2=S\setminus \overline D_1$.
Assume, by contradiction, that $\{x\in S: u_p^n(x)<0\}=D_1\cup G_1$ and let $G_2=\{x\in S: u_p^n(x)>0\}\neq \emptyset$. 
For $i=2,3$ let $G_i^n$ be the subset of $B\setminus \mathcal Z_{u_p^n}$ which contains $G_i$ and is $n$-invariant. We have that $D_1^n\cap G_i^n=\emptyset$ for each $i=2,3$ and this contradicts Corollary \ref{cor1}. Then 
$u_p^n$ has the same sign throughout $D_2$, namely $u_p^n>0$. \\
Next we claim that $\bar D_1\subset S\cup \{O\}\cup \{\partial B\}$
and, in particular, either $\{O\}\in \partial D_1$, or $\partial D_1\cap \partial B\neq \emptyset$ or none of the two holds. In all these cases $D_2$ is connected, $u_p^n>0$ in $\partial S\cap (B\setminus\{O\})$ so that 
$D_2^n$, the subset of $B\setminus \mathcal Z_{u_p^n}$ which contains $D_2$ and is $n$-invariant, is connected concluding the proof.\\
Assume, by contradiction, that  $\bar D_1\cap (\partial S\cap B\setminus\{O\})\neq \emptyset$. We call $t_1$ and $t_2$ the two lines which 
border the sector $S$. By symmetry reasons $\bar D_1$ intersects both the lines $t_1$ and $t_2$ in two points $x_1,x_2$ such that $|x_1|=|x_2|$. Moreover $u_p^n(x)<0$ for every $x\in S$ such that $|x|=|x_1|$. Since $D_1$ is contained in $S$, it is not possibile that $\partial D_1$ contains a part of $t_1$ or $t_2$, because this would contradict the Hopf boundary Lemma. Then there should be a part of a nodal region $G_1$ in which $u_p^n>0$ contained in $B(0,|x_1|)$ and another part $G_2$ contained in $B\setminus \bar B(0,|x_1|)$ which are separated by $D_1$. Letting as before $G_i^n$ be the subset of $B\setminus \mathcal Z_{u_p^n}$ which contains $G_i$ and is $n$-invariant, for $i=2,3$, we have that $D_1^n\cap G_i^n=\emptyset$ for $i=2,3$ and this contradicts Corollary \ref{cor1}.
Then $u_p^n>0$ on $t_1,t_2$. \\
Finally if $\{O\}\in \partial D_1$ then it is obvious, by our assumptions, that $\partial D_1\cap \partial B=\emptyset$, since $u_p^n=0$ on $\partial D_1$. It is now easy to see that $D_2$ is connected. Indeed if $\{O\}\notin \partial D_1$ there exists a neighborhood of $O$ in $S$ in which $u_p^n>0$ that connects $t_1$ and $t_2$ and the same is true whenever $\partial D_1\cap \partial B=\emptyset$. This shows that also $D_2^n$ is connected and proves that the nodal regions of $u_p^n$ in this case are $n+1$.
\end{proof}

\noindent We end this part with a final Lemma which considers the case when the closure of the nodal set of $u_p^n $ intersects the boundary $\partial B$.
\begin{lemma}\label{lem-15}
Let $u_p^n$ be a least energy $n$-invariant nodal solution to \eqref{H}. Assume that the closure of $\mathcal Z_{u_p^n}$ intersects $\partial B$. Then $u_p^n$ admits at least $n+1$ nodal regions in $B$.
\end{lemma}
\begin{proof}
If $\overline{\mathcal Z_{u_p^n}}\cap \partial B\neq \emptyset$ we can assume there exists at least a point $\bar x\in \overline{\mathcal Z_{u_p^n}}\cap \partial S_{\frac {2\pi}n}\cap \partial B$ such that $\frac {\bar y}{\bar x}\in (0,\frac {2\pi}n)$. If $\bar x\in \mathcal B$, by symmetry reasons and the Hopf Lemma, $\mathcal Z_{u_p^n}$ cannot be along $\mathcal B$ and then $\{(u_p^n)^{-1}(0)\}\cap B(\bar x, r)\cap B$ is made of at least $2$ simple arcs, one above $\mathcal B$ and one below it. Let us consider the arc below $\mathcal B$ which is contained in $S_{\frac{\pi}n}$ and is the boundary of a nodal region for $u_p^n$. Since $u_p^n=0$ on $\partial B$, either this arc is contained in $S_{\frac{\pi}n}$ and then its closure intersects $\partial B$, or it intersects $\mathcal B$ 
or it intersect the $x$ axis for $x\neq O$. In the first two cases, by the symmetry of $u_p^n$ with respect to $\mathcal B$, this arc shapes the boundary of a nodal region of $u_p^n$ which is contained in the sector $S_{\frac{2\pi}n}$ and then we can apply either Lemma \ref{lem-12} or Lemma \ref{lem-14} depending on the fact if the origin belongs to this arc or not. In both cases $u_p^n$ admits at least $n+1$ nodal regions and we are done. 
In the case when the arc intersects the $x$ axis instead, by symmetry of $u_p^n$, we have that a nodal component of $u_p^n$ is contained in $\overline{S_{\frac{\pi}n}}\cup R_{-\frac{\pi} n}(S_{\frac{\pi}n})$ which is a sector of angle $\frac {2\pi}n$ and again we can apply Lemma \ref{lem-14} getting the result.\\
Finally, if the closure of $\mathcal Z_{u_p^n}$ intersects $\partial B$ in a point which does not lie on the bisector, then there is at least one component of $\mathcal Z_{u_p^n}$ in $S_{\frac{\pi}n}$ and we can repeat exactly the previous argument to get the thesis.
\end{proof}

\noindent We can now prove Theorem \ref{teo3} and Corollaries \ref{cor-fin} and \ref{cor-fin2}.
\begin{proof}[Proof of Theorem \ref{teo3}]
Let $u_p^n$ a least energy $n$-invariant solution. Either there exists a sector $S$ of angle $\frac {2\pi}n$ that contains a nodal region of $u_p^n$ or not. In the second case all the nodal regions of $u_p^n$ are $n$-invariant and connected. This implies that $u_p^n$ admits only two nodal regions by Corollary \ref{cor1} and we are in case $3)$. Of course the nodal line cannot touch the boundary of $B$ due to Lemma \ref{lem-15}.\\
Else, if there exists a sector $S$ of angle $\frac {2\pi}n$ that contains a nodal region of $u_p^n$, then either $\mathcal Z_{u_p^n}$ admits a connected component that contains the origin and whose closure intersects $\partial B$ or not. In the first case we are in case $1)$ and by Lemma \ref{lem-12} $u_p^n$ has $2n$ nodal regions. By Proposition \ref{prop-9} this is possible if and only if $2n\le \left[\frac {2+\a}2\gamma\right]$ which yields the thesis in case $1)$. Else $\mathcal Z_{u_p^n}$ does not admit a connected component that contains the origin and whose closure intersects $\partial B$, and by Lemma \ref{lem-14} $u_p^n$ has $n+1$ nodal regions and we are in case $2)$.  By Proposition \ref{prop-9} this is possible if and only if $n+1\le\left[\frac {2+\a}2\gamma\right]$ concluding the proof.
\end{proof}

\begin{proof}[Proof of Corollary \ref{cor-fin}]
By Theorem \ref{teo2} we know that $u_p^1, u_p^2,\dots, u_p^{\lceil \frac {2+\a}2 \kappa-1\rceil}$ are distinct nodal nonradial solutions to \eqref{H}. By Theorem \ref{teo3} we know that they are quasiradial when $n>\max\{\left[\frac {2+\a}2\gamma-1\right], \left[\frac {2+\a}4\gamma\right] \}=\left[\frac {2+\a}2\gamma-1\right]$. The proof then follows observing that, by the definition of $\gamma$ and $\kappa$ and the properties of the ceiling function, for every value of $\a$, it holds $\left[\frac {2+\a}2\gamma-1\right]<\lceil \frac {2+\a}2 \kappa-1\rceil$.
Moreover the value $\lceil \frac {2+\a}2 \kappa-1\rceil-\left[\frac {2+\a}2\gamma-1\right]$ is increasing in $\a$ and, since for $\a=0$ $\lceil \kappa-1\rceil-\left[\gamma-1\right]=2$, we have at least two quasiradial solutions for every value of $\a$. 
The regularity of the nodal set follows by Corollary \ref{cor2}, since $O\notin \mathcal Z_{u_p^n}$.
\end{proof}

\begin{proof}[Proof of Corollary \ref{cor-fin2}]
The thesis follows since
\[\lceil \frac {2+\a}2 \kappa-1\rceil-\left[\frac {2+\a}2\gamma-1\right]\geq \frac{2+\a}2(\kappa-\gamma).\]
\end{proof}

\end{document}